\documentclass[a4paper,twoside,reqno]{amsart}

\usepackage[pagewise]{lineno}
\usepackage[utf8]{inputenc}

\usepackage{authblk,orcidlink}
\usepackage{hyperref}
\hypersetup{urlcolor=blue, citecolor=red}
\usepackage{amsmath,amsthm,amssymb,xcolor,bbm,mathrsfs,graphicx, float}
\usepackage[english]{babel}
\usepackage{fancyhdr}
\makeatletter
\@namedef{subjclassname@1991}{2020 Mathematics Subject Classification}
\makeatother

\title{The multilayer garbage disposal game}
\author{Hsin-Lun Li}

\date{}
\email{hsinlunl@math.nsysu.edu.tw}

\theoremstyle{definition}
\newtheorem{theorem}{Theorem}

\newtheorem{lemma}[theorem]{Lemma}

\begin{document}

\allowdisplaybreaks

\thispagestyle{firstpage}
\maketitle
\begin{center}
    Hsin-Lun Li \orcidlink{0000-0003-1497-1599}
    \centerline{$^1$National Sun Yat-sen University, Kaohsiung 804, Taiwan}
\end{center}
\medskip

\begin{abstract}\sloppy
The multilayer garbage disposal game is an evolution of the garbage disposal game. Each layer represents a social relationship within a system of finitely many individuals and finitely many layers. An agent can redistribute their garbage and offload it onto their social neighbors in each layer at each time step. We study the game from a mathematical perspective rather than applying game theory. We investigate the scenario where all agents choose to average their garbage before offloading an equal proportion of it onto their social neighbors. It turns out that the garbage amounts of all agents in all layers converge to the initial average of all agents across all layers when all social graphs are connected and have order at least three. This implies that the winners are those agents whose initial total garbage exceeds the average total garbage across all agents.

\end{abstract}

\section{Introduction}
A garbage disposal game consists of a finite number of agents, each of whom processes a garbage amount, represented as a nonnegative real number. Each agent updates their garbage either by offloading it onto others or by receiving garbage from others at each time step~\cite{hirai2006coalition}. Interpreted mathematically, consider an $n$-individual garbage disposal game. Let $x_i(t)$ denote the garbage amount of individual $i$ at time $t$, and let $M_n(S)$ represent the collection of all $n \times n$ matrices whose entries belong to the set $S$. The evolution rule is given by:
$$x(t+1) = A(t)x(t),$$
where
\begin{align*}
    x(t) &= (x_1(t), \ldots, x_n(t))^T= \text{transpose of } (x_1(t), \ldots, x_n(t)), \\
    A(t) &\in M_n(\mathbb{R}_{\geq 0}) \ \text{is column stochastic.}
\end{align*}

Assume that garbage represents bad things, whether abstract, such as unhappiness, or tangible, such as excrement, with a norm assigned to quantify each as a value. This allows for the calculation of an individual's total garbage amount. Under different social relationships, individuals have opportunities to offload garbage onto their social neighbors or receive garbage from them. For instance, in friendships or counterparty relationships, individuals can offload garbage, such as unhappiness, onto their social neighbors or receive garbage from them. The garbage disposal game fundamentally differs from opinion dynamics~\cite{lanchier2022consensus,mHK,mHK2,li2024imitation,li2024mixed}. In opinion dynamics, an agent can modify their opinion without directly influencing others. However, in the garbage disposal game, offloading garbage onto others directly impacts their garbage amounts.

In this paper, we propose a multilayer garbage disposal game, where each layer represents a social relationship, such as a friendship. Agents can redistribute their garbage across all layers, as well as offload it onto their social neighbors within the same layer at each time step. We consider the case in which each individual averages their garbage across all layers and offloads an equal proportion of garbage onto their social neighbors in each layer. Interpreting this mathematically, consider a system of $n$ individuals with $m$ layers. Let the undirected simple graph $G_k = ([n], E_k)$ be the social graph depicting the social relationship in layer $k$, with vertex set $[n] = \{1, \ldots, n\}$ and edge set $E_k$. Edge $(i, j) \in E_k$ if agents $i$ and $j$ are \emph{socially connected} in layer $k$. Let $x_{ji}(t)$ denote the amount of garbage for agent $i$ on layer $j$ at time $t$, and let 
$$\bar{x}_i(t) = \frac{1}{m} \sum_{k \in [m]} x_{k i}(t)$$ represent the average amount of garbage for agent \(i\) across all layers at time $t$. $N_{ji}=\{k: (k,i)\in E_j\}$ is the neighborhood of agent $i$ in layer $j.$ The update rule is as follows:

\begin{equation}\label{Model:multilayer}
    x_{ji}(t+1)=\bar{x}_i(t)(1-\frac{|N_{ji}|}{|E_j|})+\frac{1}{|E_j|}\sum_{\ell\in N_{ji}}\bar{x}_{\ell}(t)\ \hbox{for all}\ i\in [n]\ \hbox{and}\ j\in [m].
\end{equation}
According to the Six Degrees of Separation concept introduced by Frigyes  Karinthy, the length of the shortest path between any two individuals in a social graph is at most six. Later, the concept of Four Degrees of Separation was proposed based on an analysis of the entire Facebook network of active users~\cite{backstrom2012four}. Hence, it is natural to assume that the social graphs are connected. We investigate the eventual garbage amount distribution of all individuals across all layers.

Let $I$ be the identity matrix and $\mathbbm{1}$ be the column vector with all entries equal to 1. A symmetric matrix $M$ is called a \emph{generalized Laplacian} of a graph $G = (V, E)$ if, for $x, y \in V$, the following two conditions hold:
\[
M_{xy} = 0 \quad \text{for} \quad x \neq y \quad \text{and} \quad (x,y) \notin E, 
\]
and
\[
M_{xy} < 0 \quad \text{for} \quad x \neq y \quad \text{and} \quad (x,y) \in E.
\]
 The Laplacian of a graph $G$ is given by $\mathscr{L} = D_G - A_G$, where $D_G$ is the degree matrix of $G$ and $A_G$ is the adjacency matrix of $G$. Therefore, a Laplacian is a generalized Laplacian. Let $\lambda_k(A)$ denote the $k$th smallest eigenvalue of a real symmetric matrix $A$.

\section{Main results}
 Theorem~\ref{Thm:multilayer} reveals that when all social graphs in all layers are connected and have an order of at least three, the garbage amounts of all agents in all layers converge to the initial average of all agents across all layers. In particular, in the connected single-layer garbage disposal game, averaging is equivalent to not averaging. The garbage amounts of all agents converge to the initial average across all agents. Specifically, the winners are those agents whose garbage amounts are above the initial average, while the losers are those whose garbage amounts are below the initial average. 

In the multilayer garbage disposal game, concentrating one's garbage in a particular layer implies that one's garbage amounts in other layers are low. When all agents choose to average their garbage in all layers before offloading it onto their social neighbors, the winners are those whose initial total garbage exceeds the average total garbage among all agents.  

\begin{theorem}\label{Thm:multilayer}
    Assume that \( G_k = ([n], E_k) \), \( k \in [m] \), are connected with \( n \geq 3 \). Then, \( x_{ji}(t) \to \frac{1}{mn} \sum_{j \in [m],\ i \in [n]} x_{ji}(0) \) as \( t \to \infty \) for all \( j \in [m] \) and \( i \in [n] \).
\end{theorem}

\section{The multilayer garbage disposal game}
Since~\eqref{Model:multilayer} in matrix form involves Laplacians, the following lemmas outline properties of symmetric matrices and Laplacians, which are instrumental in deriving the main results.
\begin{lemma}\label{lemma:eigenvalue}
    Let $A \in M_n(\mathbb{R})$ be symmetric. Then for all $c_1\in \mathbb{R},$
    $$\begin{array}{ll}
    \displaystyle    \lambda_k(c_1 I - c_2 A) = c_1 - c_2 \lambda_{n+1-k}(A)& \quad \text{for}\ c_2 \geq 0,\vspace{2pt} \\
   \displaystyle \lambda_k(c_1 I - c_2 A) = c_1 - c_2\lambda_k(A) &\quad \text{for}\ c_2 < 0.
    \end{array}$$
\end{lemma}

\begin{proof}
    Observe that
\begin{align*}
    A v = \lambda v \iff \left(c_1 I - c_2 A\right)v = \left(c_1 - c_2 \lambda\right)v.
\end{align*}
Also,
\begin{align*}
    &c_1 - c_2\lambda_1(A) \geq \ldots \geq c_1 - c_2\lambda_n(A) \quad \text{for } c_2 \geq 0, \\
    &c_1 - c_2\lambda_1(A) \leq \ldots \leq c_1 - c_2\lambda_n(A) \quad \text{for } c_2 < 0.
\end{align*}
This completes the proof.

\end{proof}

\begin{lemma}[\cite{das2003improved}]\label{lemma:largest eigenvalue of a Laplacian}
    The largest eigenvalue $\lambda_{\max}$ of the Laplacian of a simple undirected graph $G = (V, E)$ satisfies
    \begin{equation*}
        \lambda_{\max} \leq \max\{|N_i \cup N_j| : i, j \in V \ \text{and}\ (i, j) \in E\},
    \end{equation*}
    where $N_i = \{j \in V : (i, j) \in E\}$ is the set of all neighbors of vertex $i$ in $G$.
\end{lemma}

\begin{lemma}[Perron-Frobenius for Laplacians \cite{biyikoglu2007laplacian}]
\label{Lemma:PF_for_Laplacians}
 Assume that~$M$ is a generalized Laplacian of a connected graph.
 Then, the smallest eigenvalue of~$M$ is simple and the corresponding eigenvector can be chosen with all entries positive.
\end{lemma}

\begin{lemma}[\cite{horn2012matrix}]\label{lemma:sublinear}
   Let $A$, $B \in M_n(\mathbb{R})$ be symmetric. Then,
\[
\begin{array}{ll}
    \displaystyle \lambda_i(A + B) \leq \lambda_{i+j}(A) + \lambda_{n-j}(B), & \quad j = 0, 1, \dots, n - i \vspace{2pt} \\
    \displaystyle \lambda_i(A+B) \geq \lambda_{i-j+1}(A) + \lambda_j(B), & \quad j = 1, \ldots, i
\end{array}
\]
for all $i \in [n].$
\end{lemma}


\begin{proof}[\bf Proof of Theorem~\ref{Thm:multilayer}]
Let \( x_{jt} = (x_{j1}(t), \ldots, x_{jn}(t)) \), \( x_t = (x_{1t}, \ldots, x_{mt})^T \), and \( B_j = I - \frac{1}{|E_j|} \mathscr{L}_j \), where \( \mathscr{L}_j \) is the Laplacian on \( G_j = ([n], E_j) \). Letting
\[
B = \begin{bmatrix}
B_1 & B_1 & \cdots & B_1 \\
B_2 & B_2 & \cdots & B_2 \\
\vdots & \vdots & \ddots & \vdots \\
B_m & B_m & \cdots & B_m
\end{bmatrix} \in M_{mn}(\mathbb{R}),
\]
the mechanism of~\eqref{Model:multilayer} in matrix form is
\[
x_{t+1} = \frac{1}{m} B x_t \quad \text{for all} \quad t \geq 0, \quad \text{so} \quad x_t = \frac{B^t}{m^t} x_0.
\]
Observe that
\[
B^t = \begin{bmatrix}
B_1 \left( \sum_{i \in [m]} B_i \right)^{t-1} & B_1 \left( \sum_{i \in [m]} B_i \right)^{t-1} & \cdots & B_1 \left( \sum_{i \in [m]} B_i \right)^{t-1} \\
B_2 \left( \sum_{i \in [m]} B_i \right)^{t-1} & B_2 \left( \sum_{i \in [m]} B_i \right)^{t-1} & \cdots & B_2 \left( \sum_{i \in [m]} B_i \right)^{t-1} \\
\vdots & \vdots & \ddots & \vdots \\
B_m \left( \sum_{i \in [m]} B_i \right)^{t-1} & B_m \left( \sum_{i \in [m]} B_i \right)^{t-1} & \cdots & B_m \left( \sum_{i \in [m]} B_i \right)^{t-1}
\end{bmatrix}.
\]
Write
\[
\sum_{i \in [m]} B_i = m \left[ I - \frac{1}{m} \sum_{i \in [m]} \frac{\mathscr{L}_i}{|E_i|} \right] = m \left[ I - \frac{1}{m \prod_{i \in [m]} |E_i|} \sum_{i \in [m]} \frac{\prod_{j \in [m]} |E_j|}{|E_i|} \mathscr{L}_i \right].
\]
Note that
\[
\sum_{i \in [m]} \frac{\prod_{j \in [m]} |E_j|}{|E_i|} \mathscr{L}_i
\]
is a Laplacian on a connected multigraph. Since the Laplacians are positive semidefinite, by Lemma~\ref{Lemma:PF_for_Laplacians}, 
\[
\lambda_1 \left( \sum_{i \in [m]} \frac{\prod_{j \in [m]} |E_j|}{|E_i|} \mathscr{L}_i \right) = 0
\]
is simple, so by Lemma~\ref{lemma:eigenvalue},
\begin{align*}
    &\lambda_n \left( I - \frac{1}{m \prod_{i \in [m]} |E_i|} \sum_{i \in [m]} \frac{\prod_{j \in [m]} |E_j|}{|E_i|} \mathscr{L}_i \right) \\
    &\hspace{1cm}= 1 - \frac{1}{m \prod_{i \in [m]} |E_i|} \lambda_1 \left( \sum_{i \in [m]} \frac{\prod_{j \in [m]} |E_j|}{|E_i|} \mathscr{L}_i \right) \\
    &\hspace{1cm}= 1 \quad \text{is simple.}
\end{align*}
Via Lemmas~\ref{lemma:eigenvalue},~\ref{lemma:sublinear}, and~\ref{lemma:largest eigenvalue of a Laplacian}, we have
\begin{align*}
    \lambda_1 \left( I - \frac{1}{m} \sum_{i \in [m]} \frac{\mathscr{L}_i}{|E_i|} \right) &= 1 - \frac{1}{m} \lambda_n \left( \sum_{i \in [m]} \frac{\mathscr{L}_i}{|E_i|} \right) \\
    &\geq 1 - \frac{1}{m} \sum_{i \in [m]} \frac{\lambda_n (\mathscr{L}_i)}{|E_i|} \\
    &\geq 1 - \frac{n}{\min_{i \in [m]} |E_i|} \\
    &\geq 1 - \frac{n}{n - 1} = -\frac{1}{n - 1} > -1 \quad \text{for} \quad n \geq 3.
\end{align*}
Therefore,
\[
-1 < \lambda_k \left( I - \frac{1}{m} \sum_{i \in [m]} \frac{\mathscr{L}_i}{|E_i|} \right) < 1 \quad \text{for} \quad 1 \leq k \leq n - 1.
\]
Letting
\[
C = I - \frac{1}{m} \sum_{i \in [m]} \frac{\mathscr{L}_i}{|E_i|},
\]
we can write
\[
C = P \, \text{diag} \left( 1, \lambda_{n-1}(C), \ldots, \lambda_1(C) \right) P^T
\]
where \( P \in M_n(\mathbb{R}) \) is orthonormal. Note that \( \frac{\mathbbm{1}}{\sqrt{n}} \) is the orthonormal eigenvector of \( C \) with eigenvalue 1 and is independent of \( t \). So,
\[
C^{t-1} \to P \, \text{diag} \left( 1, 0, \ldots, 0 \right) P^T = \frac{1}{n} \mathbbm{1} \mathbbm{1}^T \quad \text{as} \quad t \to \infty.
\]
Hence,
\[
\frac{1}{m^t} B_j \left( \sum_{j \in [m]} B_j \right)^{t-1} = \frac{1}{m} B_j C^{t-1} \to \frac{1}{nm} B_j \mathbbm{1} \mathbbm{1}^T = \frac{1}{nm} \mathbbm{1} \mathbbm{1}^T \quad \text{as} \quad t \to \infty.
\]
Thus,
\[
x_{ji}(t) \to \frac{1}{nm} \sum_{i \in [n]; j \in [m]} x_{ji}(0) \quad \text{as} \quad t \to \infty.
\]

\end{proof}
\section{Statements and Declarations}
\subsection{Competing Interests}
The author is partially funded by NSTC grant.

\subsection{Data availability}
No associated data was used.

\end{document}